\documentclass[10pt]{article}
\usepackage{amscd, enumerate, amsmath, amssymb,amsthm}
\usepackage[dvips]{color}

\newtheorem{defn}{Definition}[section]
\newtheorem{thm}[defn]{Theorem}
\newtheorem{prop}[defn]{Proposition}
\newtheorem{lem}[defn]{Lemma}
\newtheorem{rem}[defn]{Remark}
\newtheorem{ex}[defn]{Example}

\numberwithin{equation}{section}

\newcommand{\lam}{\lambda}
\newcommand{\F}{{\mathcal F}}
\newcommand{\E}{{\mathcal E}}
\newcommand{\N}{{\mathcal N}}

\newcommand{\R}{{\mathbb R}}

\newcommand{\cH}{{\mathcal H}}

\newcommand{\la}{\langle}
\newcommand{\ra}{\rangle}

\newcommand{\1}{{\bf 1}}
\newcommand{\0}{{\bf 0}}

\title{Gram matrices of reproducing kernel Hilbert spaces over graphs}

\author{
{\sc Michio SETO}
\thanks{The author was supported by Grant-in-Aid for Young Scientists (B) (23740106).}\\
[1ex]
{\small Shimane University, 
Matsue 690-8504, Japan} \\
{\small 
{\it E-mail address}: {\tt mseto@riko.shimane-u.ac.jp}}\\
\\
{\sc Sho SUDA}
\thanks{The author was supported
by JSPS Research Fellowships for Young Scientists.}\\
[1ex]
{\small International Christian University,
Mitaka 181-8585, Japan} \\
{\small 
{\it E-mail address}: {\tt p001883e@nt.icu.ac.jp}}\\
\\
{\sc Tetsuji TANIGUCHI}\\
[1ex]
{\small 
Matsue College of Technology, 
Matsue 690-8518, Japan} \\
{\small 
{\it E-mail address}: {\tt tetsuzit@matsue-ct.ac.jp}}
}

\begin{document}

\maketitle

\begin{abstract}
In this paper, we introduce the notion of reproducing kernel Hilbert spaces for graphs and the Gram matrices associated with them.
Our aim is to investigate the Gram matrices of reproducing kernel Hilbert spaces. 
We provide several bounds on the entries of the Gram matrices of reproducing kernel Hilbert spaces and characterize the graphs which attain our bounds.

\end{abstract}

\section{Introduction}\label{Introduction}
The theory of reproducing kernel Hilbert spaces is one of richest areas in functional analysis. 
It originated from Sturm-Liouville theory in ordinary differential equations and Cauchy's integral formula in complex analysis, 
and serves nice frameworks to many fields in pure mathematics, 
applied mathematics, statistics and engineering (see \cite{S1,S2}).  
The purpose of this paper is to investigate graphs via theory of reproducing kernel Hilbert spaces. 
Since this paper is directed to mathematicians not only in functional analysis but also in the graph theory, 
we would like to include details of our basic idea. 
\begin{defn}
A vector space $\cH$ is called a reproducing kernel Hilbert space 
{\rm(}which will be abbreviated to RKHS{\rm )} over some set $\Omega$ if 
\begin{enumerate}
\item $\cH$ is a Hilbert space consisting of functions on $\Omega$, 
\item for any $x$ in $\Omega$, 
there exists a non-zero function $k_x$ in $\cH$ such that $f(x)=\la f,k_x \ra_{\cH}$ for any function $f$ in $\cH$, 
where $\la\cdot,\cdot\ra_{\cH}$ denotes the inner product of $\cH$. 
\end{enumerate}
\end{defn}
The above $k_x$ is called the reproducing kernel of $\cH$ at $x$. 
Then the set of reproducing kernels $\{k_x\}_{x\in \Omega}$ is a linearly independent dense subset of $\cH$. 
Setting $k(x,y)=\la k_y,k_x\ra_{\cH}$, which defines a two variable function on $\Omega\times \Omega$. 
In particular, $K=(k(x,y))_{x,y}$ can be identified with a self-adjoint matrix 
if $\Omega$ is a finite set, and which is called the Gram matrix of $\cH$.
In this paper, we construct real RKHS's encoding data of graphs, 
and investigate their Gram matrices. 
As will be shown in Section~\ref{sec:Gramregular}, the Gram matrix of $\cH$ coincides with the inverse of the sum of the Laplacian matrix and the all-ones matrix.
This matrix has been extensively studied in the context of the generalized inverse matrix of the Laplacian matrix, see \cite{B1,B2} and references therein. 
In fact, the main results of this paper are Theorems~\ref{thm:1000}, \ref{thm:1001} and \ref{thm:112}.
These results provide bounds on the entries of the Gram matrices of reproducing kernel spaces of graphs and characterize graphs which attain our bounds.

Before stating our results,  
we should mention the work of Nagai, Kametaka, Yamagishi, Takemura and Watanabe in \cite{NKYTW}, 
where some similar results have been obtained. 
One of differences from their research is that our interest is  analysis of graphs from RKHS point of view.  

The paper is organized as follows. 
Section~\ref{Preliminary} is a short introduction to the graph theory which we need in this paper. 
In Section~\ref{RKHS}, we introduce Sobolev type real RKHS's over graphs, which will be denoted $\cH_G$. 
In Section~\ref{gram}, we deal with Gram matrices of  $\cH_G$. 
An estimation of spectra of Gram matrices is given. 
In Section~\ref{sec:Gramregular}, we show that the Gram matrix is the inverse of the sum of the Laplacian matrix and the all-ones matrix.
It turns out that reciprocal of the positive eigenvalues of the Laplacian matrix of the graph are eigenvalues of the Gram matrix and eigenspaces of both matrices coincide.
Finally in Section~\ref{sec:entry} we focus on the maximum and minimum value of entries of the Gram matrix $K$.
We provide the upper and lower bounds for diagonal entries of the Gram matrix of a graph and 
the upper and lower bound for minimum entries of the Gram matrix of a tree.
In both cases, we characterize the graph which attains each bound.

\section{Preliminary}\label{Preliminary}
A graph $G=(V(G),E(G))$ is a pair of a nonempty finite set $V(G)$, which is called the vertex set, and a subset $E(G)$ in $\{\{x,y\}: x,y \in V \text{ and } x\neq y\}$, which is called the edge set.
In this paper, a graph means always simple, namely it has neither loops nor mulitiedges, and has finite vertices. 
We abbreviate $V=V(G)$ and $E=E(G)$ if there is no confusion.
For any vertex $i\in V$, let $d_i$ or $\deg_{G}(i)$ denote the degree $|\{j\in V:\{i,j\}\in E\}|$.
A path from $x$ to $y$ in the graph $G$ is a sequence of $x_0=x,x_1,\ldots,x_l=y$ such that $\{x_i,x_{i+1}\}\in E(G)$ for $0\leq i\leq l-1$ and the vertices $x_0,\ldots,x_l$ are all distinct.   
The graph $G$ is said to be connected if, for any two distinct vertices in $G$, there exists a path from one to the other.
A graph is said to be a tree if, for any two distinct vertices, there exists the unique path between them. 
Let $d=d_G$ denotes the path-length distance for a connected graph $G$.
Graphs $G_1=(V(G_1),E(G_1))$ and $G_2=(V(G_2),E(G_2))$ are isomorphic if there exists a bijection $\Phi:V(G_1)\rightarrow V(G_2)$ such that $(x,y)\in E(G_1)$ if and only if $(\Phi(x),\Phi(y))\in E(G_2)$.

The adjacency matrix of $G$ is a square matrix $A$ whose rows and columns are indexed by $V$ with $(x,y)$-entry $1$ if $\{x,y\}\in E$ and $0$ otherwise.
The degree matrix $D$ of $G$ is a diagonal matrix with $(x,x)$-entry equal to its degree.
The Laplacian matrix $L$ of $G$ is defined to be $L=D-A$.
Let $\lambda_j$ ($1\leq j\leq s$) be all distinct eigenvalues of $L$ with increasing order $\lambda_1<\cdots<\lambda_s$.
Let $m_i$ be the multiplicity of $\lambda_i$. 
It is known that $L$ has the smallest eigenvalue $0$ with multiplicity $1$ and its eigenvector is the all-ones vector $\1$ provided that $G$ is connected.

Let $E_j$ denotes the projection onto the eigenspace corresponding to $\lambda_j$ for $1\leq j\leq s$.
Then $E_1=\frac{1}{n}J$ holds, where $n$ denotes the number of vertices and $J$ denotes the all-ones matrix. 
Since $L$ is symmetric, $L$ has the following spectral decomposition; $L=\sum_{j=1}^s \lambda_j E_j$ with 
$E_j^T=E_j$, $E_i E_j=\delta_{i j}$ and $\sum_{j=1}^s E_j=I$, where $I$ denotes the identity matrix and $\delta_{i j}$ denotes the Kronecker delta.

\section{\mathversion{bold}RKHS $\cH_G$}\label{RKHS}
Let $G$ be a connected graph with adjacency matrix $A$. 
The set of all real valued functions on $V$ will be denoted by $\F=\F(G)$.  
We define a bilinear form on $\F$ as follows: 
$$\E(u,v)=\frac{1}{2}\sum_{x,y\in V}A_{x y}(u(x)-u(y))(v(x)-v(y)).$$

\begin{lem}[Cauchy-Schwarz]\label{lem:3-1}
$|\E(u,v)|^2\leq \E(u,u)\E(v,v)$. 
\end{lem}

\begin{proof}\rm
It is the same as the usual inner product case. 
\end{proof}

\begin{lem}\label{lem:3-2} 
$\E(u,u)=0$ if and only if $u$ is constant on $V$.  
\end{lem}

\begin{proof}\rm
The only if part is trivial. 
We shall show the if part. 
Assume that  
$$\frac{1}{2}\sum_{x,y\in V}A_{x y}|u(x)-u(y)|^2=0.$$
Then it follows that $u(x)=u(y)$ if $A_{x y}\neq 0$. 
Since $G$ is connected, $u$ is constant. 
\end{proof}

We set $\N=\R \1$ and $([u],[v])=\E(u,v)$ for $[u],[v]$ in $\F/\N$. 
Note that $([u],[v])$ is well defined. 
Indeed, 
if $[u]=[u']$ and $[v]=[v']$, then we have that 
\begin{align*}
|\E(u,v)-\E(u',v')|
&\leq |\E(u,v)-\E(u,v')|+|\E(u,v')-\E(u',v')|\\
&=|\E(u,v-v')|+|\E(u-u',v')|\\
&\leq \{(\E(u,u)\E(v-v',v-v')\}^{1/2}+\{\E(u-u',u-u')\E(v,v)\}^{1/2}\\
&=0
\end{align*}
by Cauchy-Schwarz inequality (Lemma \ref{lem:3-1}). 
We define another bilinear form on $\F$ as follows: 
$$\la u,v \ra=(\sum_{x\in V}u(x))(\sum_{x\in V}v(x))+([u],[v])$$
for 
$u$ and $v$ in $\F$. 

\begin{lem}\label{lem:3-3}
$\la\cdot, \cdot\ra$ is an inner product invariant under graph isomorphisms. 
\end{lem}

\begin{proof}\rm 
We set $c=\sum_{x\in V}u(x)$. 
If $\la u,u \ra=0$ then $|c|^2=0$ and $([u],[u])=0$. 
Then, by Lemma \ref{lem:3-2}, $u$ is constant. 
Therefore we have that $u=0$. 
Thus $\la\cdot,\cdot\ra$ is an inner product. 
Next, 
let $\Phi$ be a graph isomorphism from $G_1$ onto $G_2$, and 
let $\la \cdot,\cdot\ra_1$ and $\la\cdot ,\cdot\ra_2$ be 
the corresponding inner products of $G_1$ and $G_2$, respectively. 
Then, since $A_{x y}=A_{\Phi^{-1}(x)\Phi^{-1}(y)}$, we have that 
\begin{align*}
&\quad\la u\circ \Phi^{-1} ,u\circ \Phi^{-1} \ra_2\\
&=|\sum_{x\in V(G_2)}u(\Phi^{-1}(x))|^2
+\frac{1}{2}\sum_{x,y\in V(G_2)}A_{x y}|u(\Phi^{-1}(x))-u(\Phi^{-1}(y))|^2\\
&=|\sum_{x\in V(G_2)}u(\Phi^{-1}(x))|^2
+\frac{1}{2}\sum_{x,y\in V(G_2)}A_{\Phi^{-1}(x) \Phi^{-1}(y)}|u(\Phi^{-1}(x))-u(\Phi^{-1}(y))|^2\\
&=|\sum_{x'\in V(G_1)}u(x')|^2+\frac{1}{2}\sum_{x',y'\in V(G_1)}A_{x'y'}|u(x')-u(y')|^2\\
&=\la u,u\ra_1. 
\end{align*}
By the polarization identity, 
it follows that inner product $\la\cdot,\cdot\ra$ is 
invariant under graph isomorphisms. 
This concludes the proof. 
\end{proof}

\begin{defn}
$\cH_G$ will denote the Hilbert space $(\F,\la\cdot,\cdot\ra)$. 
The norm induced by $\la \cdot,\cdot \ra$ will be denoted by $\|\cdot\|_{\cH_G}$, 
that is, we set 
$$\|u\|_{\cH_G}^2=|\sum_{x\in V}u(x)|^2+\E(u,u).$$
$\cH_G$ will be abbreviated as $\cH$ if no confusion occurs. 
\end{defn}

Since the dimension of $\cH$ is finite, 
the point evaluation on $\cH$ is norm continuous. 
By Riesz representation theorem, 
there exists a non-zero element $k_x$ in $\cH$ such that $u(x)=\la u,k_x \ra$ for any $u$ in $\cH$. 
We note that $k_x$ is uniquely determined. 
This $k_x$ is called the reproducing kernel of $\cH$ at $x$. 
In this paper, $k_x(y)$ will be denoted by $k(y,x)$. 
It is easy to see that $\{k_x\}_{x\in V}$ is linearly independent. 

\begin{defn}
Let $G_1$ and $G_2$ be connected graphs.  
We will say that $\cH_{G_1}$ and $\cH_{G_2}$ are isomorphic as reproducing kernel Hilbert spaces if 
there exists a unitary operator $U$ from $\cH_{G_1}$ onto $\cH_{G_2}$,  
a graph isomorphism $\Phi$ from $G_1$ onto $G_2$ and 
a function $\varphi$ on $V(G_2)$ such that the following diagram commutes: 
$$
\begin{CD} 
\cH_{G_1} @>{U}>> \cH_{G_2}\\
@A{\iota_1}AA @AA{\frac{1}{\varphi}\iota_2}A\\
G_1 @>>{\Phi}> G_2,
\end{CD}
$$
where $\iota$ denote the mapping $x\mapsto k_x$ from $V(G)$ into $\cH_G$. 
For the fully general definition, see {\rm \cite{AM}}. 
\end{defn}

The next lemma is well known in the theory of reproducing kernel Hilbert spaces. 

\begin{lem}\label{lem:3-4}
Let $k$ and $j$ be reproducing kernels of $\cH_{G_1}$ and $\cH_{G_2}$, respectively. 
If $\cH_{G_1}$ and $\cH_{G_2}$ are isomorphic, 
then
\begin{enumerate}
\item $j_{\Phi(x)}(y)=\varphi(\Phi(x))\varphi(y)k_x(\Phi^{-1}(y))$, 
\item $(Uk_x)(y)=\varphi(y)k_x( \Phi^{-1}(y))$, 
\item $Uf=\varphi(f\circ \Phi^{-1})$. 
\end{enumerate}
\end{lem}

\begin{proof}\rm
By diagram chasing. 
\end{proof}

Let $\delta_x$ denote the delta function for a vertex $x$ in $V$. 
If $\{x,y\}$ belongs to $E$, then $\delta_x$ is orthogonal to $\delta_y$ in $\cH_G$. 
Indeed, for $\{x,y\}$ in $E$, we have that  
\begin{align*}
\la \delta_x,\delta_y \ra
&=1+\E(\delta_x,\delta_y)\\
&=1+\frac{1}{2}\{(\delta_x(x)-\delta_x(y))(\delta_y(x)-\delta_y(y))+(\delta_x(y)-\delta_x(x))(\delta_y(y)-\delta_y(x))    \}\\
&=0.
\end{align*}
Hence, if $G$ is complete, then $\{\delta_x\}_{x\in V}$ is an orthogonal set in $\cH_G$, and 
we have that 
\begin{align*}
k_x
&=\sum_{y\in V}\la k_x,\frac{1}{\|\delta_y\|_{\cH}}\delta_y \ra_{\cH}\frac{\delta_y}{\|\delta_y\|_{\cH}}\\
&=\sum_{y\in V}\frac{1}{\|\delta_y\|_{\cH}^2}\delta_y(x)\delta_y\\
&=\frac{1}{|V|}\delta_x.
\end{align*}

\begin{thm}\label{lem:3-5} 
Let $G_1$ and $G_2$ be non-complete graphs. 
If $\cH_{G_1}$ and $\cH_{G_2}$ are isomorphic, then $\varphi\equiv 1$ or $\varphi\equiv -1$.  
\end{thm}

\begin{proof}\rm 
Since $\|\delta_x\|_{\cH_{G_1}}^2=1+\deg_{G_1}(x)$, we have
\begin{align*}
\|U\delta_x\|_{\cH_{G_2}}^2
&=\|{\varphi}(\delta_x\circ \Phi^{-1})\|_{\cH_{G_2}}^2\\
&=\|{\varphi}\delta_{\Phi(x)}\|_{\cH_{G_2}}^2\\
&=|\varphi(\Phi(x))|^2+\E({\varphi}\delta_{\Phi(x)},{\varphi}\delta_{\Phi(x)})\\
&=\|{\varphi(\Phi(x))}\delta_{\Phi(x)}\|_{\cH_{G_2}}^2\\
&=|\varphi(\Phi(x))|^2\{1+\deg_{G_2}(\Phi(x))\}
\end{align*}
by Lemma \ref{lem:3-4}. 
Since $U$ is unitary and $\deg_{G_1}(x)=\deg_{G_2}(\Phi(x))$, 
it follows that $|\varphi|=1$ on $V(G_2)$. 
Next, let $\{x,y\}$ be not in $E(G_1)$. 
Without loss of generality, we may assume that $\varphi(\Phi(x))=1$ 
and $\varphi(\Phi(y))=a$. 
Then 
$\la \delta_x,\delta_y \ra_{\cH_{G_1}}=1$ and 
\begin{align*}
\la U\delta_x,U\delta_y\ra_{\cH_{G_2}}
&=\la {\varphi}(\delta_x\circ \Phi^{-1}), {\varphi}(\delta_y\circ \Phi^{-1})\ra_{\cH_{G_2}}\\
&=\la {\varphi(\Phi(x))}\delta_{\Phi(x)}, {\varphi(\Phi(y))}\delta_{\Phi(y)}   \ra_{\cH_{G_2}}\\
&={\varphi(\Phi(x))} \varphi(\Phi(y))\\
&=a
\end{align*}
by Lemma \ref{lem:3-4}. 
Therefore we have that $a= 1$. 
This concludes the proof. 
\end{proof}

\begin{thm}\label{thm:3-2}
If $G_1$ and $G_2$ are isomorphic, then $\cH_{G_1}$ and $\cH_{G_2}$ are isomorphic. 
\end{thm}

\begin{proof}\rm 
Let $\Phi$ be a graph isomorphism from $G_1$ onto $G_2$. 
Setting $Uu=u\circ \Phi^{-1} $, $U$ is a unitary operator 
from $\cH_{G_1}=(\F(G_1),\la \cdot,\cdot \ra_{\cH_{G_1}})$ onto $\cH_{G_2}
=(\F(G_2),\la \cdot,\cdot \ra_{\cH_{G_2}})$ by Lemma \ref{lem:3-3}, 
and then trivially $U^{\ast}v=v\circ \Phi$.  
In order to show the statement, we shall see that $Uk_x=j_{\Phi(x)}$ for any $x$ in $V(G_1)$. 
Let $v$ be in $\cH_{G_2}$. Then we have that 
$$\la v, Uk_x \ra_{\cH_{G_2}}=\la U^{\ast}v , k_x\ra_{\cH_{G_1}}=\la v\circ \Phi,k_x \ra_{\cH_{G_1}}
=v(\Phi(x))=\la v, j_{\Phi(x)} \ra_{\cH_{G_2}}.$$
This concludes the proof. 
\end{proof}

\begin{rem}\rm
For a non-connected graph $G$, 
we define $\cH_G$ as 
the Hilbert space direct sum $\cH_{G_1}\oplus \cdots \oplus \cH_{G_n}$,  
where $G_1,\ldots,G_n$ are connected components of $G$.   
\end{rem}

\section{Gram matrices}\label{gram}

We set $\|u\|_2=(\sum_{x\in V}|u(x)|^2)^{1/2}$ for each $u$ in $\F$, 
and $l^2(G)$ will denote the Hilbert space over $G$ with the norm $\|\cdot \|_2$.  

\begin{lem}\label{lem:4-1}
Let $G$ be a connected graph with $n$ vertices, 
and let $T$ be the operator from $\cH_G$ into $l^2(G)$ defined as 
$Tu=(u(x_1),\ldots, u(x_n))$. 
Then 
\begin{enumerate}
\item $TT^{\ast}=(k_{x_j}(x_i))_{ij}$
\item $T^{\ast}T=k_{x_1}\otimes k_{x_1}+\cdots +k_{x_n}\otimes k_{x_n}$, 
where we set $(u\otimes v)w=\la w,v \ra u$. 
\end{enumerate}
\end{lem}

\begin{proof}\rm
For any $(a_1,\ldots,a_n)$ in $\R^n$,  we have that   
\begin{align*}
\la Tu,(a_1,\ldots ,a_n) \ra_{l^2}
&=\la (u(x_1),\ldots ,u(x_n)),(a_1,\ldots a_n) \ra_{l^2}\\
&=u(x_1)a_1+\cdots+ u(x_n)a_n\\
&=\la u,a_1k_{x_1}+\cdots +a_nk_{x_n} \ra_{\cH}.
\end{align*}
Hence we have that 
$$T^{\ast}(a_1,\ldots,a_n)=a_1k_{x_1}+\cdots +a_nk_{x_n}.$$
It follows that 
\begin{align*}
TT^{\ast}(a_1,\ldots,a_n)
&=T(a_1k_{x_1}+\cdots +a_nk_{x_n})\\
&=(a_1k_{x_1}(x_1)+\cdots +a_nk_{x_n}(x_1), \ldots, a_1k_{x_1}(x_n)+\cdots +a_nk_{x_n}(x_n))\\
&=(k_{x_j}(x_i))(a_1,\ldots, a_n).
\end{align*}
This concludes (i). Moreover, 
\begin{align*}
T^{\ast}Tu
&=T^{\ast}(u(x_1),\ldots,u(x_n))\\
&=u(x_1)k_{x_1}+\cdots +u(x_n)k_{x_n}\\
&=(k_{x_1}\otimes k_{x_1}+\cdots +k_{x_n}\otimes k_{x_n})u.
\end{align*}
Thus we have (ii). 
\end{proof}

\begin{defn}
We set $k(y,x)=k_{x}(y)$ and $K=TT^{\ast}=(k(x_i,x_j))_{ij}$. 
In this paper, $K$ will be called the Gram matrix of $G$. 
The $(i,j)$-entry of $K$ is also denoted by $K_{i,j}$.
\end{defn}

\begin{thm}\label{thm:4-1}
Let $G_1$ and $G_2$ be connected graphs, and 
let $K_1$ and $K_2$ be Gram matrices of 
$G_1$ and $G_2$, respectively.  
If $\cH_{G_1}$ and $\cH_{G_2}$ are isomorphic, then $K_1=K_2$. 
\end{thm}

\begin{proof}\rm 
Let $k$ and $j$ be reproducing kernels of $\cH_{G_1}$ and $\cH_{G_2}$, respectively. 
By Lemmas \ref{lem:3-4} and \ref{lem:3-5}, we have that 
$$k(y,x)=\la k_x,k_y\ra_{\cH_{G_1}}
=\la Uk_{x},Uk_{y} \ra_{\cH_{G_2}}
=\la j_{\Phi(x)},j_{\Phi(y)} \ra_{\cH_{G_2}}=j(\Phi(y),\Phi(x)).
$$
This concludes the proof. 
\end{proof}

\begin{thm}\label{thm:4-2}
Let $G$ be a connected graph. Then $1/|V|\in \sigma(K)$. 
\end{thm}

\begin{proof}\rm 
Since $|V|\sum_y k_x(y)
=\la k_x,\1 \ra_{\cH_G}=1$, 
we have $K\1=\1/|V|$. 
\end{proof}

\begin{rem}\rm
In the later section, we will see that 
$\min \sigma(K)=1/|V|.$
\end{rem}

\begin{lem}\label{lem:4-2}
Let $G$ be a connected graph with $n$ vertices. Then
$$\sigma(K)\subset \left\{ \|u\|_2^2/\|u\|^2_{\cH_G}:u\neq 0 \right\}.$$
\end{lem}

\begin{proof}\rm 
Let $u$ be an eigenvector with respect to an eigenvalue $\lam$ of $T^{\ast}T$. 
Then we have that 
\begin{align*}
\lam \|u\|^2_{\cH}
&=\la \lam u,u \ra_{\cH}\\
&=\la T^{\ast}Tu,u\ra_{\cH} \\
&=\la Tu,Tu \ra_{l^2}\\
&=\la (u(x_1),\ldots,u(x_n)), (u(x_1),\ldots,u(x_n))\ra_{l^2}\\
&=\|u\|_2^2. 
\end{align*}
Hence we have that 
$$\sigma(T^{\ast}T)\subset \left\{ \|u\|_2^2/\|u\|^2_{\cH}:u\neq 0 \right\}.$$
Since $\sigma(TT^{\ast})=\sigma(T^{\ast}T)$, we have the conclusion. 
\end{proof}

We set $\|u\|_{\infty}=\max_{x\in V}|u(x)|$ for each $u$ in $\F$, 
and $l^{\infty}(G)$ will denote the Banach space over $G$ with the norm $\|\cdot \|_{\infty}$.

\begin{lem}\label{lem:4-3}
Let $G$ be a connected graph. Then 
$$\|u\|_{\infty}^2\leq\{\max_{x\in V}k(x,x)\} \|u\|_{\cH_G}^2,$$ and which is best possible. 
\end{lem}

\begin{proof}\rm
Let $S$ be the operator from $\cH$ onto $l^{\infty}(G)$ defined as $Su=u$. 
It suffices to show that $\|S\|^2=\max_{x\in V}k(x,x)$. 
Since 
$$|u(y)|^2=|\la u, k_y\ra |_{\cH}^2\leq \|u\|_{\cH}^2\|k_y\|_{\cH}^2,$$
we have that 
$$\|u\|_{\infty}^2\leq \{\max_{x\in V}k(x,x)\}\|u\|_{\cH}^2.$$
Hence we have that $\|S\|^2\leq \max_{x\in V}k(x,x)$. 
Conversely, taking a vertex $x_0$ in $V$ such that $k(x_0,x_0)=\max_{x\in V}k(x,x)$, 
we have that 
$$\max_{x\in V}k(x,x)\leq \|k_{x_0}\|_{\infty}\leq \|S\|\|k_{x_0}\|_{\cH}= \|S\|\{\max_{x\in V}k(x,x)\}^{1/2}.$$
It follows that $$\max_{x\in V}k(x,x)\leq \|S\|^2.$$
This concludes the proof. 
\end{proof}

\begin{thm}\label{thm:4-3}
Let $G$ be a connected graph. 
Then 
$$\max \sigma(K)\leq |V|\max_{x\in V}k(x,x).$$
\end{thm}

\begin{proof}\rm 
By Lemma \ref{lem:4-3}
we have that $$\|u\|_2^2\leq |V|\|u\|_{\infty}^2\leq |V|\{\max_{x\in V}k(x,x) \}\|u\|_{\cH}^2.$$
This concludes the proof by Lemma \ref{lem:4-2}. 
\end{proof}

\begin{rem}\rm 
The inequality in Lemma \ref{lem:4-3} 
is a kind of  discrete Sobolev inequality (cf.\ \cite{NKYTW}). 
\end{rem}

\begin{rem}\rm 
If $G$ is complete, then $\sigma(K)=\{1/|V|\}$. 
Indeed, $$\|u\|_{\cH}^2=\|\sum_{x\in V}u(x)\delta_x\|_{\cH}^2=\sum_{x\in V}|u(x)|^2\|\delta_x\|_{\cH}^2=|V|\|u\|_2^2.$$
\end{rem}

\begin{rem}\rm
For each non-connected graph $G$, 
we define the Gram matrix of $G$ as 
the direct sum $K_{G_1}\oplus \cdots \oplus K_{G_n}$,  
where $G_1,\ldots,G_n$ are connected components of $G$.   
\end{rem}

\section{The Gram matrix in terms of the spectra of Laplacian matrix}\label{sec:Gramregular}
In this section, we study the spectra of the Gram matrix $K$ of $\cH$.
\begin{thm}\label{G1}\rm
Let $G$ be a connected graph with $n$ vertices and Laplacian matrix $L$.
Assume that the spectral decomposition of $L$ is $\sum_{j=1}^s \lambda_j  E_j$ with $\lambda_1=0$.
Then $K=\frac{1}{n^2}J+\sum_{j=2}^s \frac{1}{\lambda_j}E_j$.
\end{thm}
\begin{proof}\rm
Define $\alpha_1=\frac{1}{\sqrt{n}}$ and $\alpha_j=\frac{1}{\sqrt{\lambda_j}}$ for $2\leq j\leq s$ and
$\hat{u}=\sum_{j=1}^s \alpha_j E_j \delta_u$.
Then  a set $\{\hat{u}:u\in V\}$ is an orthonormal basis with respect to the inner product $\langle, \rangle$ defined in Section~\ref{RKHS}.
Indeed
\begin{align*}
\langle \hat{u},\hat{v} \rangle
=\frac{1}{n}+\hat{u}^T L\hat{v}
=\sum_{j=1}^s (E_j)_{u v} 
=\delta_{u v}.
\end{align*}

Thus 
$
K=F F^T=\sum_{j=1}^s \alpha_j^2 E_j=\frac{1}{n^2}J+\sum_{j=2}^s \frac{1}{\lambda_j}E_j$.
\end{proof}

In particular, $K$ is the inverse matrix of $L+J$ and $K-\frac{1}{n^2}J$ is the Moore-Penrose pseudo-inverse of the Laplacian matrix $L$.
Since the largest eigenvalue of $L$ is at most $n$, $\frac{1}{n}$ is the smallest eigenvalue of $K$.

\begin{ex}\rm
Let $n$ be an integer at least $2$.
Let $G=(V,E)$ be a star with $V=\{1,\cdots,n\}$ and $E=\{\{1,i\}:  2\leq i\leq n\}$.
Then the eigenvalues of the Laplacian matrix $L$ are $\lambda_1=0$, $\lambda_2=1$ and $\lambda_3=n$.
The corresponding multiplicities are $m_1=1$, $m_2=n-2$ and $m_3=1$.
The orthogonal projections $\mathbb{R}^n$ onto the eigenspaces are $E_1=\frac{1}{n}J$, $E_3=u u^T$ where $u=\frac{1}{\sqrt{n(n-1)}}(n-1,-1,\ldots,-1)^T$ and $E_2$ is determined from the equation $\sum_{i=1}^3 E_i=I$.
Thus, by Theorem~\ref{G1}, the Gram matrix of reproducing kernel of $G$ is given by \begin{align*}
K=\frac{1}{n^2}J+E_2+\frac{1}{n}E_3
=\begin{pmatrix}
\frac{1}{n} & \0^T \\  
\0 & I-\frac{1}{n}J
\end{pmatrix},
\end{align*}
where $\0$ denotes the all-zeros column vector.
\end{ex}

\begin{ex}\label{ex:path}\rm
Let $n$ be an integer at least $2$.
Let $G=(V,E)$ be a path of length $n$ with $V=\{1,\cdots,n\}$ and $E=\{\{i,i+1\}:  1\leq i\leq n-1\}$.
Then the eigenvalues of the Laplacian matrix $L$ are $\lambda_j=4\sin^2(\frac{\pi j}{2n})$ for $0\leq j\leq n-1$.
The corresponding normalized column eigenvector for $\lambda_j$ are $u_j=(\sqrt{\frac{2}{n}}\cos((i-\frac{1}{2})\frac{\pi j}{n}))_{1\leq i\leq n}$.
The orthogonal projection $\mathbb{R}^n$ onto the the eigenspace for $\lambda_j$ is $E_j:=u_j u_j^T$.
Thus, by Theorem~\ref{G1}, the Gram matrix of reproducing kernel of $G$ is given by $K=\frac{1}{n^2}J+\sum_{j=1}^{n-1}\frac{1}{\lambda_j}E_j$.
For $1\leq i,j\leq n$, the $(i,j)$-entry of $K$ is
\begin{align*}
K_{i,j}&=\frac{1}{n^2}+\frac{1}{2n}\sum_{l=1}^{n-1}\frac{\cos((2i-1)\frac{\pi l}{2n})\cos((2j-1)\frac{\pi l}{2n})}{\sin^2(\frac{\pi l}{2n})}\\
&=\frac{(n+1)(2n^2-5n+6)}{6n^2}+\frac{(i-1)i}{2n}+\frac{(j-1)j}{2n}-\frac{|i-j|+i+j-2}{2}.
\end{align*}
\end{ex}

\begin{prop}\label{prop:022}
Let $ G $ be a connected graph with $ n $ vertices with Laplacian matrix $L$.
Let $G'$ be a graph obtained by an edge added to $G$. 
Let $ 0=\lambda_1,\lambda_2,\ldots,\lambda_s $ be all the distinct eigenvalues of $ L $ and 
$m_i$ the multiplicity of $\lambda_i$ for $1\leq i\leq s$.
Then the following hold:
\begin{enumerate}
\item $\det(K)=\frac{1}{n\lambda_2^{m_2}\cdots\lambda_s^{m_s}}=\frac{1}{n^2\tau (G)}$.
\item $\det(K(G)) \ge \det(K(G'))$.
\item $\frac{1}{n^n} \le \det(K) \le \frac{1}{n^2}$ holds.
The equality on the right holds if and only if $ G $ is a tree, and 
the equality on the left holds if and only if $ G $ is the complete graph.
\end{enumerate}
\end{prop}
\begin{proof}
Since $ \frac{1}{n},\frac{1}{\lambda_2},\cdots,\frac{1}{\lambda_s} $
	are the eigenvalues of $ K $,
	$ \det(K)=\frac{1}{n\lambda_2^{m_2}\cdots\lambda_s^{m_s}} $.
Also the fact $ \tau (G)=\frac{1}{n}\lambda_2^{m_2}\cdots\lambda_s^{m_s} $ is well-known.
Hence (i) holds.

(ii) follows immediately from (i).

 From (i) and that $\tau (G)\ge 1 $ with equality if and only if $ G $ is a tree and $\tau (G)\le n^{n-2} $ with equality if and only if $ G $ is the complete graph,
(iii) follows.
\end{proof}

\section{Entries of Gram matrix}\label{sec:entry}
In this section, we study the entries of the Gram matrix $K$ for a graph.
We will show a combinatorial interpretation and several bounds on the entries of $K$. 

Define $r_{i,j}=K_{i,i}+K_{j,j}-2K_{i,j}$.
The values $r_{i,j}$ is called the resistance between vertices $i$ and $j$, see \cite[Section 9]{B3}.
The following is a fundamental result on the resistance.
\begin{lem}[\cite{KR}]\label{lm:002}
Let $ G $ be a connected graph with $ n $ vertices and Laplacian matrix $L$.
Then the following hold:
\begin{enumerate}
\item $\{r_{i,j}:1\leq i,j\leq n\}$ satisfies the axiom of distance,	
\item $r_{i,j}\leq d(i,j)$ with equality if and only if there is the unique path between $i$ and $j$.
\end{enumerate}
\end{lem}

We use the following lemma to provide a combinatorial interpretation on the entries of $K$.
For vertices $i,j$ in a connected graph $G$, denote by $\tau(G;i,j)$ the number of spanning forests of $G$ with $2$ components in which each component contains exactly one of $\{i,j\}$. 
Thus when $i$ and $j$ coincide, $\tau(G;i,j)$ is zero.
Define $r_{G,i}=\sum_{j\in V(G)}r_{i,j}$ for each vertex $i\in V(G)$.
It is also denoted by $r_i$.
\begin{lem}\label{lem:5-1}
Let $G$ be a connected graph with $n$ vertices. 
Then the following hold.
\begin{enumerate}
\item $r_{i,j}=\frac{\tau(G;i,j)}{\tau(G)}$.
\item $r_{i}=n K_{i,i}+\text{Tr}(K)-\frac{2}{n}$.\label{eq:1100}
\item $K_{i,i}=\frac{1}{n}(\sum_{l=1}^n r_{i,l}-\frac{1}{2n}\sum_{k,l=1}^n r_{k,l}+\frac{1}{n})$.
\end{enumerate} 
\end{lem}
\begin{proof}
(i): When $i=j$, there is nothing to prove. 
Assume that $i$ and $j$ are distinct. 
From \cite[equation (9.3) in p 113]{B3}, $r_{i,j}=\frac{\det(L(i,j|i,j))}{\det(L(i|i))}$.
Also \cite[Theorem~4.7]{B3} implies that $\frac{\det(L(i,j|i,j))}{\det(L(i|i))}=\frac{\tau(G;i,j)}{\tau(G)}$. 

(ii): By the definition of $r_i$ and $r_{i,j}$, we have the following equalities:
\begin{align*}
r_i&=\sum_{j\in V(T)}r_{i,j}\\
&=\sum_{j\in V(T)}(K_{i,i}+K_{j,j}-2K_{i,j})\\
&=n K_{i,i}+\sum_{j\in V(T)}K_{j,j}-2\sum_{j\in V(T)}K_{i,j}\\
&=n K_{i,i}+\text{Tr}(K)-\frac{2}{n}.
\end{align*}

(iii): Summing (ii) over $i=1,\ldots,n$ yields that $\text{Tr}(K)=\frac{1}{2n}\sum_{i,j=1}^n r_{i,j}+\frac{1}{n}$. 
Substituting this equation (ii) we have the desired equality.
\end{proof}
Now we obtain the following equality.
\begin{prop}
Let $G$ be a connected graph with $n$ vertices.
Then the following hold.
\begin{align*}
K_{i,j}=\frac{1}{2n}\sum_{l=1}^n \frac{\tau(G;i,l)+\tau(G;j,l)}{\tau(G)}-\frac{1}{2n^2}\sum_{k,l=1}^n \frac{\tau(G;k,l)}{\tau(G)}+\frac{1}{n^2}-\frac{1}{2}\frac{\tau(G;i,j)}{\tau(G)}.
\end{align*}
\end{prop}
\begin{proof} 
The case for $i=j$ follows from Lemma~\ref{lem:5-1}(i), (iii).
The case for $i\neq j$ follows from the case for $i=j$, the definition of $r_{i,j}$ and Lemma~\ref{lem:5-1}(i).
\end{proof}

Next we derive several inequalities on the entries of $K$.
A vertex in $G$ is said to be dominating if it is adjacent to all other vertices. 
The following lemma gives a lower bounds on the diagonal entries and characterizes dominating vertices.
\begin{lem}\label{lm:1010}
Let $ G $ be a connected graph with $ n $ vertices.
For any $ i\in\{1,\ldots,n\} $,
	$ K_{i,i}\ge \frac{1}{d_i+1} $.
In particular, $K_{i,i}\geq \frac{1}{n}$ with equality if and only if $i$ is a dominating vertex in $G$.
\end{lem}
\begin{proof}
See \cite[Theorem 3.2]{F} for the proof of $ K_{i,i}\ge \frac{1}{d_i+1} $.

Since $d_1+1\leq n$, $K_{i,i}\geq \frac{1}{n}$ holds. 
If equality holds, we observe that $K_{i,i}=\frac{1}{n}$ and $K_{l,i}=0$ for any $l\neq i$.
Then $L_{i,i}=n-1$ holds, and thus a vertex $i$ is a dominating vertex.
Conversely assume that a vertex $i$ is a dominating vertex.
Then $L_{i,i}=n-1$ and $L_{i,l}=-1$ for any $l\neq i$.
Comparing the $i$-th row of $K(L+J)=I$ yields that $K_{i,i}=\frac{1}{n}$ and $K_{i,l}=0$ for any $l\neq i$.
\end{proof}

The following proposition gives an upper bound for the least entry of each row in $K$.
\begin{prop}\label{lm:1008}
Let $ G $ be a connected graph with $ n $ vertices.
If $ K_{i,j}=\min_{l\neq i}K_{i,l} $,
then $K_{i,j}\le 0$ with equality if and only if the vertex $i$ is a dominating vertex in $G$.
\end{prop}
\begin{proof}
Since all-ones vector is the eigenvector of $K$ with eigenvalue $\frac{1}{n}$, we have  
\begin{align*}
\frac{1}{n}=K_{i,i}+\sum_{l\neq i}K_{i,l}\ge K_{i,i}+(n-1)K_{i,j}.
\end{align*}
Since $K_{i,i}\geq \frac{1}{d_i+1} \geq \frac{1}{n}$ by Lemma~\ref{lm:1010}, we have the desired inequality.

The equivalent condition the equality follows from the same proof in Lemma~\ref{lm:1010}.
\end{proof}

\begin{lem}\label{lem:0000}
Let $G$ be a connected graph which does not have an edge $e=\{i,j\}$.
Let $G'$ be the graph added the edge $e$ to $G$.
Then for any $l\in\{1,\ldots,n\}$, $K_{l,l}\geq K'_{l,l}$ holds. 
\end{lem}
\begin{proof}
Letting $x=\delta_i-\delta_j$, we have $L'=L+xx^T$.
By the Sherman-Morrison formulae \cite[p.19]{HJ}, we have $K'=K-\frac{Kx x^T K}{1+x^T K x}$.
From Lemma~\ref{lm:002}(ii), 
the $(l,l)$-entry of $\frac{Kx x^T K}{1+x^T Kx}$, which equals to $\frac{(K_{l,i}-K_{l,j})^2}{1+K_{i,i}+K_{j,j}-2K_{i,j}}$, is nonnegative.
Thus $K_{l,l}\geq K'_{l,l}$ holds.
\end{proof}

\begin{lem}\label{lem:1002}
Let $T$ be a tree with $n$ vertices, $K$ the Gram matrix of $G$.
Then the following hold:
\begin{enumerate}
\item $K_{i,i}=\frac{1}{n}(r_i-\frac{1}{n}\sum_{1\leq j<m\leq n}d_T(j,m)+\frac{1}{n})$.\label{eq:1000}
\item Candidates which take the maximum value of $K_{i,i}$ are leaves on $T$.
\item Let $P=i_1\cdots i_l$ be the unique path from $i_1$ to $i_l$. Then $K_{i_1,i_1}>K_{i_1,i_2}>\cdots>K_{i_1,i_l}$.
In particular, candidates which take the minimum value of $K_{i,j}$ are leaves on $T$.
\end{enumerate} 
\end{lem}
\begin{proof}
(i): Follows from Lemma~\ref{lem:5-1}(ii) and Lemma~\ref{lm:002}(ii).

(ii): Let $i,j$ be adjacent vertices in $T$.
Let $V_1$ (respectively $V_2$) be a set of vertices whose distance from $i$ (respectively $j$) is greater than that from $j$ (respectively $i$).  
Then $V(T)=V_1\cup V_2 \cup \{i,j\}$ and 
\begin{align*}
r_i&=\sum_{x\in V(T)}r_{i,x}\\
&=\sum_{x\in V(T)}d(i,x)\\
&=\sum_{x\in V_1}d(i,x)+\sum_{x\in V_2}d(i,x)+d(i,j)\\
&=\sum_{x\in V_1}d(i,x)+\sum_{x\in V_2}(d(j,x)+1)+1\\
&=\sum_{x\in V_1}d(i,x)+\sum_{x\in V_2}d(j,x)+|V_2|+1.
\end{align*}
Similarly $r_j=\sum_{x\in V_1}d(i,x)+\sum_{x\in V_2}d(j,x)+|V_1|+1$.
Thus, by (ii), $K_{i,i}\geq K_{j,j}$ if and only if $|V_1|\leq |V_2|$.
Using this equivalent, if $K_{i,i}$ takes the maximum value in $\{K_{l,l}:1\leq l\leq n\}$ then $i$ is a leaf.

(iii): Let $i$ be a fixed vertex.
Let $j$ be a leaf which is not equal to $i$, and $j'$ be the unique vertex adjacent to $j$.
Then comparing the $(i,j)$-entry of $K L=I-\frac{1}{n}J$ yields $K_{i,j}-K_{i,j'}=-\frac{1}{n}$.
Thus $K_{i,j}<K_{i,j'}$ holds.
Next let $j$ be a vertex in $T$ which is not equal to $i$, and $j'$ be the unique vertex adjacent to $j$ such that $d_T(i,j')=d_T(i,j)-1$.
Assume that, for any vertex $l$ such that $j$ lies on the unique path between $i$ and $l$ and $d_T(i,l)>d_T(i,j)$, $K_{i,j}>K_{i,l}$.
Then comparing the $(i,j)$-entry of $K L=I-\frac{1}{n}J$ yields $d_j K_{i,j}-K_{i,j'}-\sum_{l\sim j,m\neq j}K_{i,l}=-\frac{1}{n}$.
Thus $K_{i,j'}-K_{i,j}=\sum_{l\sim j,m\neq j}(K_{i,j}-K_{i,l})+\frac{1}{n}>0$. 
This proves (iv).
\end{proof}

The following lemma will be used in the proof of Theorems~\ref{thm:1000} and \ref{thm:1001}. 
\begin{lem}\label{lem:1003}
Let $T$ be a tree with $n$ vertices.
Let $i$ be a leaf in $T$ and $j$ the unique vertex adjacent to $i$.
Let $T'$ be the tree obtained by deleting the leaf $i$ from $T$.
Then the following hold:
\begin{enumerate}
\item $\sum_{l\in V(T)}d_T(i,l)=\sum_{l\in V(T')}d_{T'}(j,l)+n-1$,
\item $\sum_{l,m\in V(T)}d_T(l,m)=\sum_{l,m\in V(T')}d_{T'}(l,m)+2\sum_{l\in V(T')}d_{T'}(j,l)+2(n-1)$.
\end{enumerate}
\end{lem}
\begin{proof}
Both equalities follow from direct calculations.
\end{proof}

Define $\overline{K}=\overline{K}(G)$ to be the maximum value on the diagonal entries of $K$ for $G$.
\begin{thm}\label{thm:1000}
Let $G$ be a connected graph with $n$ vertices, $K$ the Gram matrix of $G$.
Then $\frac{1}{n}\leq \overline{K}\leq \overline{K}(P_n)$ with left equality if and only if $G$ is the complete graph and with right equality if and only if $G$ is the path.
\end{thm}
\begin{proof}
By Lemma~\ref{lem:0000}, it is easy to see that $\overline{K}$ takes the minimum if and only if $G$ is the complete graph.

For the right inequality, we prove theorem by induction on $n$.
For $n=2$, there is nothing to prove.
We assume that the statement holds for $n-1\geq 2$.
By Lemma~\ref{lem:0000}, it is enough to consider the case that $G$ is a tree $T$ and the vertices taking the maximum are leaves on the tree.
Let $1$ be a leaf on the tree $T$ and $2$ the vertex adjacent to $1$ in $T$.
Let $T'$ be the tree with vertices $\{2,\ldots,n\}$ in $T$ with Gram matrix $K'$.
Substituting the equations in Lemma~\ref{lem:1003}(i) and (ii) into the equation in Lemma~\ref{lem:1002}(i), we obtain 
\begin{align*}
K_{1,1}
&=\frac{(n-1)^2}{n^2}K'_{2,2}+\frac{(n-1)^2}{n^2}.
\end{align*}
The value takes the maximum if and only if $T'$ is the path with $n-1$ vertices and the vertex $2$ in $T'$ is a leaf by the assumption for induction.
Thus $K_{1,1}$ takes the maximum if and only if $T$ is the path with $n$ vertices and the vertex $1$ in $T$ is a leaf.
\end{proof}

Define $\underline{K}=\underline{K}(G)$ to be the minimum value on the entries of $K$ for $G$.
\begin{thm}\label{thm:1001}
Let $T$ be a tree with $n$ vertices, $K$ the Gram matrix of $T$.
Then $\underline{K}(P_n)\leq\underline{K}\leq 0$ with left equality if and only if $T$ is the path and with right equality if and only if $T$ is the star.
\end{thm}
\begin{proof}
By Lemma~\ref{lem:1002}(iii), the minimum in $K_{i,j}$ for all $1\leq i,j\leq n$ takes at distinct $i,j$.
Applying Proposition~\ref{lm:1008} to the case of trees implies the right inequality and characterizes the graph attaining the bound.

For the left inequality, we prove theorem by induction on $n$.
For $n=2,3$, there is nothing to prove.
We assume that the statement holds for $n-2\geq 2$.
By Lemma~\ref{lem:1002}(iii), it is enough to consider the case that the vertices are  leaves on the tree.

Let $1,n$ be leaves on the tree $T$ and $2$ (respectively $n-1$) the vertex adjacent to $1$ (respectively $n$) in $T$.
Let $T'$ be the tree with vertices $\{2,\ldots,n-1\}$ in $T$ with Gram matrix $K'$.
Substituting equalities obtained by using Lemma~\ref{lem:1003} two times into the equation in Lemma~\ref{lem:1002}(i) and using Lemma~\ref{lm:002}(ii), we obtain 
\begin{align*}
K_{1,n}
&=\frac{(n-2)^2}{n^2}K'_{2,n-1}-\frac{n-1}{n^2}d_{T'}(2,n-1)-\frac{2(n-1)}{n^2}.
\end{align*}
The value $K_{1,n}$ takes the minimum if and only if $K'_{2,n-1}$ takes the minimum and $d_{T'}(2,n-1)$ takes the maximum.
By the assumption of the induction, this condition is equivalent that $T'$ is the path with $n-2$ vertices and the vertices $2,n-1$ in $T'$ are leaves.
Thus $K_{1,n}$ takes the minimum if and only if $T$ is the path with $n$ vertices and the vertices $1,n$ in $T$ are leaves.
\end{proof}

\begin{rem}\rm
All entries of $K$ for the path of length $n$ are explicitly given in Example~\ref{ex:path}. 
\end{rem}

Finally we give the lower bound of the entries of vertices with distance $p-1$ for trees for $p\geq2$.
The following lemma will be used in the proof of Theorem~\ref{thm:112}.  
\begin{lem}\label{lem:111}
Let $n,p$ be integers such that $n\geq p\geq 2$.
Let $T$ be a tree with $n$-vertices and $1,2,\cdots,p$ a path with length $p-1$ in $T$.
Let $T_i$ denote the tree of $T-\cup_{k=1}^{p-1}\{k,k+1\}$ containing $i$ with $n_i$ vertices for each $1\leq i\leq p$. 
Then the following hold:
\begin{enumerate}
\item $r_{T,1}=\sum_{k=1}^p(r_{T_k,k}+(k-1)n_k)$ and $r_{T,p}=\sum_{k=1}^p(r_{T_k,k}+(p-k)n_k)$.
\item 
$
\sum_{(l,m)\in V(T)^2}d(l,m)=\sum_{i=1}^p \sum_{(l,m)\in V(T_i)^2}d(l,m)+2\sum_{i=1}^p(n-n_i)r_{T_i,i}+\sum_{i,j=1}^p|i-j|n_i n_j
$.
\item $K_{1,p}\geq \frac{1}{n}-\frac{p}{n^2}+\frac{1}{n^2}-\frac{\sum_{i,j=1}^p|i-j|n_i n_j}{2n^2}$ with equality if and only if each $T_i$ is a star with the dominating vertex $i$.
\end{enumerate}
\end{lem}
\begin{proof}
(i), (ii): Follow from direct calculations.

(iii): From (i) and (ii) above and Lemma~\ref{lm:1010}, 
\begin{align*}
K(T)_{1,p}
&=\frac{1}{2n}(r_1+r_p-\frac{1}{n}\sum_{(l,m)\in V(T)^2}d(l,m))+\frac{1}{n^2}-\frac{p-1}{2}\\
&=\frac{1}{n^2}\sum_{i=1}^p n_i(r_{T_i,i}-\frac{1}{2 n_i}\sum_{(l,m)\in V(T_i)^2}d(l,m))+\frac{1}{n^2}-\frac{\sum_{i,j=1}^p|i-j|n_i n_j}{2n^2}\\
&=\frac{1}{n^2}\sum_{i=1}^p n_i^2K(T_i)_{i,i}-\frac{p}{n^2}+\frac{1}{n^2}-\frac{\sum_{i,j=1}^p|i-j|n_i n_j}{2n^2}\\
&\geq \frac{1}{n^2}\sum_{i=1}^p n_i-\frac{p}{n^2}+\frac{1}{n^2}-\frac{\sum_{i,j=1}^p|i-j|n_i n_j}{2n^2}\\
&=\frac{1}{n}-\frac{p}{n^2}+\frac{1}{n^2}-\frac{\sum_{i,j=1}^p|i-j|n_i n_j}{2n^2}.
\end{align*}
Equality holds if and only if $K(T_i)_{i,i}=\frac{1}{n_i}$ for each $i$.
The latter condition is equivalent that $T_i$ is a star with the dominating vertex $i$ by Lemma~\ref{lm:1010}.
This proves (iii).
\end{proof}
As a consequence of Lemma~\ref{lem:111}, we obtain the lower bound for the entires of the Gram matrix corresponding to vertices with distance $p-1$.
Denote by $T^*_{r,s,p}$ the tree of order $r+s+p-2=n$, which is obtained by joining the dominating vertices of two stars with $r$ and $s$ vertices by an path with length $p-1$. 
\begin{thm}\label{thm:112}
Let $T$ be a tree with $n$-vertices.
Let $1,p$ be vertices with length $p-1$ in $T$.
Then $K_{1,p}\geq \frac{1}{n}-\frac{p}{n^2}+\frac{1}{n^2}
-\frac{p-1}{n^2}(\frac{(3n-2p+3)(p-2)}{6}+\lceil \frac{n-p+2}{2}\rceil\lfloor \frac{n-p+2}{2} \rfloor)$ 
with equality if and only if $T$ is a tree $T^*_{\lceil \frac{n-p+2}{2}\rceil, \lfloor \frac{n-p+2}{2} \rfloor,p}$ and $1$ and $p$ are the vertices with degree $\lceil\frac{n-p+2}{2}\rceil,\lfloor \frac{n-p+2}{2} \rfloor$.  
\end{thm}
\begin{proof}
By Lemma~\ref{lem:111}(iii), we have $K_{1,p}\geq \frac{1}{n}-\frac{p}{n^2}+\frac{1}{n^2}-\frac{\sum_{i,j=1}^p|i-j|n_i n_j}{2n^2}$.
We show that, under the conditions $\sum_{i=1}^p n_i=n$ and each $n_i$ is a positive integer, $\sum_{i,j=1}^p|i-j|n_i n_j$ takes maximum only when
$\{n_1,n_p\}=\{\lceil \frac{n-p+2}{2}\rceil, \lfloor \frac{n-p+2}{2} \rfloor\}$ and $n_2=\cdots =n_{p-1}=1$ by induction on $n$ and $p$.

For $n=p$, there is nothing to prove. 
Set $f_p(n_1,\ldots,n_p)=\sum_{i,j=1}^p|i-j|n_i n_j$.
Let $\sigma$ be a permutation on $\{1,\ldots,p\}$ 
such that 
$n_{\sigma(1)}\geq n_{\sigma(2)}\geq\cdots \geq n_{\sigma(\lfloor \frac{p}{2} \rfloor)}$ and $n_{\sigma(\lceil \frac{p}{2} \rceil)}\leq n_{\sigma(\lceil \frac{p}{2} \rceil+1)}\leq\cdots \leq n_{\sigma(p)}$.
Then it follows from the rearrangement inequality that $f_p(n_1,\ldots,n_p)\leq f_p(n_{\sigma(1)},\ldots,n_{\sigma(p)})$. 
Thus the statement is true for $n=p+1$.

Next we consider the case $p=2$.
In this case, it is easily verified that $f_2(n_1,n_2)$ takes maximum only when $\{n_1,n_p\}=\{\lceil \frac{n}{2}\rceil, \lfloor \frac{n}{2} \rfloor\}$.

Assume the statement is true for $(n-2,p)$ with $n-2\geq p$ and $(n-1,p-1)$ with $n-1\geq p-1\geq 2$.
We use the following equality:
\begin{align}
f_p(n_1,n_2,\ldots,n_{p-1},n_p)&=f_p(n_1-1,n_2,\ldots,n_{p-1},n_p-1)-2(p-1)(\sum_{i=1}^p n_i-1)\nonumber \\ 
&=f_p(n_1-1,n_2,\ldots,n_{p-1},n_p-1)-2(p-1)(n-1).\label{eq:6-13}
\end{align}
By the above argument, we may assume that $n_{1}\geq n_{2}\geq\cdots \geq n_{\lfloor \frac{p}{2} \rfloor}$ and 
$n_{\lceil \frac{p}{2} \rceil}\leq n_{\lceil p \rceil+1}\leq\cdots \leq n_{p}$.   
If $n_p=1$ holds, then $n_1\geq 2$ and by \eqref{eq:6-13} 
\begin{align*}
f_p(n_1,n_2,\ldots,n_{p-1},n_p)&=f_p(n_1-1,n_2,\ldots,n_{p-1},n_p-1)-2(p-1)(n-1)\\
&=f_{p-1}(n_1-1,n_2,\ldots,n_{p-1})-2(p-1)(n-1).
\end{align*}
Thus $f_p(n_1,n_2,\ldots,n_{p-1},n_p)$ takes maximum if and only if $f_{p-1}(n_1-1,n_2,\ldots,n_{p-1})$ takes maximum.
Then, in particular, $\{n_1-1,n_{p-1}\}=\{\lceil \frac{n-p+2}{2}\rceil, \lfloor \frac{n-p+2}{2} \rfloor\}$ holds.
However this case cannot attain the maximum value because of the fact that maximum is taken for $n_{1}\geq n_{2}\geq\cdots \geq n_{\lfloor \frac{p}{2} \rfloor}$ and $n_{\lceil \frac{p}{2} \rceil}\leq n_{\lceil \frac{p}{2}+1 \rceil}\leq \cdots n_{p}$.
Similarly the case $n_1=1$ yields the same conclusion.

Thus we consider the case $n_1\geq 2, n_p\geq 2$.
Then by \eqref{eq:6-13} and induction, $f_p(n_1-1,n_2,\ldots,n_{p-1},n_p-1)$ takes maximum only when $\{n_1-1,n_p-1\}=\{\lceil \frac{n-p}{2}\rceil,\lfloor \frac{n-p}{2} \rfloor\}$ and  $n_2=\cdots=n_{p-1}=1$. 
thus the statement is true for $n$.
Equality holds if and only if $\{n_1,n_p\}=\{\lceil \frac{n-p+2}{2}\rceil, \lfloor \frac{n-p+2}{2} \rfloor\}$ and $n_2=\cdots=n_{p-1}=1$ hold, equivalently the tree is $T^*_{\lceil \frac{n-p+2}{2}\rceil, \lfloor \frac{n-p+2}{2} \rfloor,p}$. 
This completes the proof. 
\end{proof}

\begin{rem}\rm
In \cite{Z2,Z1,ZW}, the matrix $(L+I)^{-1}$, called a doubly stochastic graph matrix, has been extensively studied.
The upper bound for the diagonal entries of a doubly stochastic graph matrix was given in \cite{Z1}, the upper bound and the lower bound for the entries of a doubly stochastic graph matrix were given in \cite{ZW} 
and the lower bound for entries with adjacent vertices for trees was given in \cite{Z2}.
Also a graph which attains each of bounds above is characterized.
The results above and Theorems~\ref{thm:1000},\ref{thm:1001} and \ref{thm:112} look like very similar, but the proofs are completely different.
\end{rem}

\end{document}